\documentclass[12pt]{article}
\usepackage{amsmath,amstext, amsthm,amsfonts,amssymb,amsthm}
 \numberwithin{equation}{section}
 \def\z{\mathbb{z}} 
  
\usepackage{epsfig}
\newtheorem{thm}{Theorem}[section]
\newtheorem{lem}[thm]{Lemma}

\newtheorem{cor}[thm]{Corollary}

\newcommand{\be}{\begin{equation}}
\newcommand{\ee}{\end{equation}}

\newcommand{\ba}{\begin{array}}
\newcommand{\ea}{\end{array}}

\newcommand{\bg}{\begin{gathered}}
\newcommand{\eg}{\end{gathered}}

\newcommand{\bea}{\begin{eqnarray}}
\newcommand{\eea}{\end{eqnarray}}

\newcommand{\Sum}{\sum_{n=0}^\infty}

\title{On Certain Generalizations of  Rogers-Ramanujan Type Identities }
\author{Ahmad El-Guindy\thanks{Research supported by a QNRF grant from Qatar Foundation NPRP No. : 7-1360-1-254
} \and   Mourad E.H. Ismail\thanks{Research supported by a QNRF grant from Qatar Foundation NPRP No. : 7-1360-1-254
}  }

\begin{document}
\date{}
 \maketitle
\begin{abstract}
 We state and prove a number of unilateral and bilateral $q$-series identities and explore some of their consequences. Those include certain generalizations of the $q$-binomial sum which also generalize the $q$-Airy function introduced by Ramanujan, as well as certain identities with an interesting variable-parameter symmetry based on limiting cases of Heine's transformation of basic hypergeomteric functions.
 \end{abstract}



\bigskip


\section{Introduction}

The theory of $q$-series is well known for a number of fascinating identities with far reaching number theoretic consequences. A famous such example comes from the Rogers-Ramanujan identities

\bea
\Sum \frac{q^{n^2}}{(q;q)_n} = \frac{1}{(q, q^4;q^5)_\infty}\label{eqRR1} \\
\Sum \frac{q^{n^2+n}}{(q;q)_n} = \frac{1}{(q^2, q^3;q^5)_\infty},
 \label{eqRR}
 \eea
where we follow the standard notations for $q$-shifted factorials and basic hypergeometric series as in the books \cite{And:Ask:Roy}, \cite{Gas:Rah}, \cite{Ismbook}.  References for the Rogers-Ramanujan identities, their origins and many of 
 their applications are in \cite{And2}, \cite{And}, and \cite{And:Ask:Roy}.  In particular we recall  
 the partition theoretic interpretation of the first Rogers--Ramanujan identity as the partitions of an 
 integer $n$ into parts $\equiv 1 \textup{ or}\; 4 \pmod{5}$ are equinumerous with the partitions of $n$ 
 into parts where any two parts differ by at least 2.  The Roger-Ramanujan identities had many extensions and generalizations to different settings. One noteworthy generalization is to extend the 
 identities in \eqref{eqRR}  to evaluate the sum $\Sum q^{n^2+mn}/(q;q)_n$, $m =0, \pm 1, \pm2, \cdots$, see \cite{Gar:Ism:Sta}. One can view the Rogers-Ramanujan identities as evaluations of Ramanujan's function defined by 
\[
A_q(z)=\sum_{n=0}^\infty \frac{q^n}{(q;q)_n}(-z)^n.
\]
Namely, \eqref{eqRR1} evaluates $A_q(-1)$, \eqref{eqRR} evaluates $A_q(-q)$ .

Another remarkable identity,  which is even simpler, is the $q$-binomial identity \cite[(II.3)]{Gas:Rah}
\begin{equation}\label{qbinom}
\sum_{n=0}^\infty \frac{(a;q)_n}{(q;q)_n} z^n=\frac{(az;q)_\infty}{(z;q)_\infty}.
\end{equation}
In recent work by Ismail and zhang \cite{Ism:zha1}, the following function was  considered
\begin{equation}
A_{q}^{\left(\alpha\right)}\left(a;t\right)=\sum_{n=0}^{\infty}\frac{\left(a;q\right)_{n}q^{\alpha n^{2}}t^{n}}{\left(q;q\right)_{n}}.\label{Aqdef} 
\end{equation}
Clearly, $A_q^{(\alpha)}(a;t)$ specializes to $A_q(z)$ with $a=0, \alpha=1$ and $t=-z$. Furthermore, it specializes to the left hand side of \eqref{qbinom} when $\alpha=0$. Naturally, one wouldn't expect a simple closed form identity for such a general function, nonetheless we present in section 2 below some generalizations of identities on $A_q^{(\alpha)}$ relating different values of the parameter $\alpha$. We also address similar generalizations to a bilateral analogue of \eqref{Aqdef} related to the Ramanujan ${}_1\psi_1$ sums  \cite[(II.29)]{Gas:Rah}
\bea
\label{eq1psi1}
\sum_{-\infty}^\infty \frac{(a;q)_n}{(b;q)_n} z^n = 
\frac{(q, b/a, az, q/az;q)_\infty}{(b, q/a, z, b/az;q)_\infty}, \quad \left|\frac{b}{a}\right| < |z| < 1. 
\eea

In section 3, we consider the series
\[
F(a,c;z):=\sum_{k=0}^\infty \frac{(a;q)_k (-1)^k q^{k(k-1)/2}z^k}{(q,c;q)_k},
\]
which can be obtained as a limit of basic hypergeometric series. We then utilize hypergeometric transformations to obtain some infinite and terminating series identites.

\section{Unilateral and Bilateral analogues of Rogers-Ramanujan Identities}

We start by recalling  the following results which were proved in \cite{Ism:zha1}. 
\begin{lem}\label{lemma 1}[8, Lemma 4.1]
For nonnegative integer $j,k,\ell,m,n$ and $\rho=e^{2\pi i/3}$ we
have 
\begin{equation}
\sum_{k=0}^{n}\frac{\left(a;q\right)_{k}\left(a;q\right)_{n-k}\left(-1\right)^{k}}{\left(q;q\right)_{k}\left(q;q\right)_{n-k}}=\begin{cases}
0 & n=2m+1\\
\frac{\left(a^{2};q^{2}\right)_{m}}{\left(q^{2};q^{2}\right)_{m}} & n=2m
\end{cases},\label{eq:multiple sum 1}
\end{equation}
 and
\begin{equation}
\sum_{\begin{array}{c}
j+k+\ell=n\\
j,k,\ell\ge0
\end{array}}\frac{\left(a;q\right)_{j}\left(a;q\right)_{k}\left(a;q\right)_{\ell}}{\left(q;q\right)_{j}\left(q;q\right)_{k}\left(q;q\right)_{\ell}}\rho^{k+2\ell}=\begin{cases}
0 & 3\nmid n\\
\frac{\left(a^{3};q^{3}\right)_{m}}{\left(q^{3};q^{3}\right)_{m}} & n=3m
\end{cases}.\label{eq:multiple sum 2}
\end{equation}
 For $j,k,m,\ell,n\in\mathbb{z}$, we have 
\begin{equation}
\sum_{j+k=n}\frac{\left(a;q\right)_{j}\left(a;q\right)_{k}\left(-1\right)^{k}}{\left(b;q\right)_{j}\left(b;q\right)_{k}}=\begin{cases}
0 & n=2m+1\\
\frac{\left(q,b/a,-b,-q/a;q\right)_{\infty}}{\left(-q,-b/a,b,q/a;q\right)_{\infty}}\frac{\left(a^{2};q^{2}\right)_{m}}{\left(b^{2};q^{2}\right)_{m}} & n=2m
\end{cases}\label{eq:multiple sum 3}
\end{equation}
 and 
\begin{equation}
\sum_{j+k+\ell=n}^{\infty}\frac{\left(a;q\right)_{j}\left(a;q\right)_{k}\left(a;q\right)_{\ell}\rho^{k+2\ell}}{\left(b;q\right)_{j}\left(b;q\right)_{k}\left(b;q\right)_{\ell}}=0\label{eq:multiple sum 4}
\end{equation}
 for $3\nmid n$, 
\begin{eqnarray}
 &  & \sum_{j+k+\ell=3m}^{\infty}\frac{\left(a;q\right)_{j}\left(a;q\right)_{k}\left(a;q\right)_{\ell}\rho^{k+2\ell}}{\left(b;q\right)_{j}\left(b;q\right)_{k}\left(b;q\right)_{\ell}}\label{eq:multiple sum 5}\\
 & = & \frac{\left(q,b/a;q\right)_{\infty}^{3}}{\left(b,q/a;q\right)_{\infty}^{3}}\frac{\left(b^{3},q^{3}a^{-3};q^{3}\right)_{\infty}}{\left(q^{3},b^{3}a^{-3};q^{3}\right)}\frac{\left(a^{3};q^{3}\right)_{m}}{\left(b^{3};q^{3}\right)_{m}}.\nonumber 
\end{eqnarray}
 \end{lem}

We start by proving an  extension of these results to general primitive roots of unity.   In order to make our notation compact; we define for $r\geq 2$ and $n\geq 0$ sets 
\bea
C_r(n):=\left\{(k_1,\dots,k_r): k_i\in \z \textrm{ and } \sum_{i=1}^r k_i=n\right\}.
\eea
Also, let $C_r^+(n)$ be the subset of $C_r(n)$ whose entries are all nonnegative. Furthermore, we let $\zeta_r$ denote the primitive $r$th root of unity $e^{2\pi i/r}$. 

\begin{lem}\label{new1}
The following identities hold for $r\geq 2$
\begin{equation}
\sum_{(k_1,\dots,k_r)\in C_r^+(n)}\frac{\left(a;q\right)_{k_1}\left(a;q\right)_{k_2}\cdots (a;q)_{k_r}}{\left(q;q\right)_{k_1}\left(q;q\right)_{k_2}\cdots (q;q)_{k_r}}\zeta_r^{\sum_{i=1}^r ik_i}=\begin{cases}
0 & r\nmid n,\\
\frac{\left(a^{r};q^{r}\right)_{m}}{\left(q^{r};q^{r}\right)_{m}} & n=rm,
\end{cases}\label{u1}
\end{equation}

\begin{equation}
\sum_{(k_1,\dots,k_r)\in C_r(n)}\frac{\left(a;q\right)_{k_1}\cdots (a;q)_{k_r}}{\left(b;q\right)_{k_1}\cdots (b;q)_{k_r}}\left(\zeta_r\right)^{\sum_{i=1}^r ik_i}=\begin{cases}
0 & r\nmid n,\\
 \frac{\left(q,b/a;q\right)_{\infty}^{r}}{\left(b,q/a;q\right)_{\infty}^{r}}\frac{\left(b^{r},q^{r}a^{-r};q^{r}\right)_{\infty}}{\left(q^{r},b^{r}a^{-r};q^{r}\right)}\frac{\left(a^{r};q^{r}\right)_{m}}{\left(b^{r};q^{r}\right)_{m}}
 & n=rm,
\end{cases}\label{b1}
\end{equation} 

 \end{lem}

\begin{proof} We start by noting that, for  $|t|<1$ we have

\[
\frac{\left(at;q\right)_{\infty}}{\left(t;q\right)_{\infty}}\frac{\left(a\zeta_r t;q\right)_{\infty}}{\left(\zeta_r t;q\right)_{\infty}}\cdots\frac{\left(a\zeta_r^{r-1}t;q\right)_{\infty}}{\left(\zeta_r^{r-1}t;q\right)_{\infty}}=\frac{\left(a^{r}t^{r};q^{r}\right)}{\left(t^{r};q^{r}\right)}, \quad \left|t\right|<1. 
\]
Employing \eqref{qbinom} we see that
\bea\label{pr1}
\prod_{i=0}^{r-1}\left( \sum_{k_i=0}^\infty \frac{(a;q)_{k_i}}{(q;q)_{k_i}} (\zeta_r^i x)^{k_i}\right) =\frac{\left(a^{r}x^{r};q^{r}\right)}{\left(x^{r};q^{r}\right)} =\sum_{m=0}^\infty \frac{(a^r;q^r)_m}{(q^r;q^r)_m}x^{rm}, \notag
\eea
and \eqref{u1} follows by comparing the coefficients of $x^n$ in \eqref{pr1}. 

The proof of \eqref{b1} is similar. We start by noting that for $\left|ba^{-1}\right|<\left|x\right|<1$, we have

\begin{eqnarray*}
\bg
\prod_{i=0}^{r-1} \frac{(q,b/a, a\zeta_r^i z, q\zeta_r^{-i}/az;q)_\infty}{(b,q/a,\zeta_r^i z, b\zeta_r^{-i}/az;q)_\infty}=\frac{(q,b/a;q)_\infty^r}{(b,q/a;q)_\infty^r}\frac{(a^rz^r, q^ra^{-r}z^{-r};q^r)_\infty}{(z^r,b^ra^{-r}z^{-r};q^r)_\infty}
  \eg
\end{eqnarray*}
applying the Ramanujan
${}_{1}\psi_{1}$ sum \eqref{eq1psi1} to that identity  establishes \eqref{b1}. 
\end{proof}

Lemma \eqref{new1} enables us to prove the following result generalizing Theorem 4.2 of \cite{Ism:zha1}
\begin{thm}
For  $\alpha\ge0$ and integer $r\geq 2$ we have 
\begin{equation}\label{e1}
A_{q^r}^{(r\alpha)}\left(a^{r};t^{r}\right)=\sum_{k_1,\dots,k_{r-1}=0}^{\infty}\frac{\left(a;q\right)_{k_1}\cdots\left(a;q\right)_{k_{r-1}}\zeta_r^{\sum_{i=1}^{r-1}ik_i}
q^{\alpha\left(\sum_{i=1}^{r-1}k_i\right)^{2}}t^{\sum_{i=1}^{r-1}k_i}}{\left(q;q\right)_{k_1}\cdots\left(q;q\right)_{k_{r-1}}}A_{q}^{\left(\alpha\right)}\left(a;q^{2\alpha(\sum_{i=1}^{r-1}k_i)}t\right).
\end{equation}

\begin{equation}\label{e2}
A_{q^r}^{(r\alpha)}\left(a^{r};t^{r}\right)=\sum_{k_2,\dots,k_r=0}^{\infty}\frac{\left(a;q\right)_{k_2}\cdots\left(a;q\right)_{k_r}\zeta_r^{\sum_{i=2}^{r-1}ik_i}
q^{\alpha\left(\sum_{i=2}^{r}k_i\right)^{2}}t^{\sum_{i=2}^{r}k_i}}{\left(q;q\right)_{k_2}\cdots\left(q;q\right)_{k_r}}A_{q}^{\left(\alpha\right)}\left(a;\zeta_rq^{2\alpha(\sum_{i=2}^{r}k_i)}t\right).
\end{equation}

\end{thm}
\begin{proof}
The proof of \eqref{e1}follows from the following series identities
\begin{equation}
\begin{aligned}&A_{q^{r}}^{\left(r\alpha\right)}\left(a^{r};t^{r}\right)=\sum_{m=0}^\infty \frac{(a^r;q^r)_m}{(q^r;q^r)_m}
 q^{\alpha r^2 m^2}t^{rm}\\
&=\sum_{n=0}^\infty \sum_{(k_1,\dots,k_r)\in C^+_r(n)} \frac{(a;q)_{k_1}\cdots (a;q)_{k_r}}{(q;q)_{k_1}\cdots(q;q)_{k_r}} \zeta_r^{\sum_{i=1}^r ik_i}t^n q^{\alpha n^2}\\
&=\sum_{k_1,\dots,k_{r-1}=0}^{\infty}\frac{\left(a;q\right)_{k_1}\cdots\left(a;q\right)_{k_{r-1}}\zeta_r^{\sum_{i=1}^{r-1}ik_i}q^{\alpha\left(\sum_{i=1}^{r-1} k_i\right)^{2}}t^{\sum_{k=1}^{r-1}k_i}}{\left(q;q\right)_{k_1}\cdots\left(q;q\right)_{k_{r-1}}}A_{q}^{\left(\alpha\right)}\left(a;tq^{2\alpha\left(\sum_{i=1}^{r-1}k_i\right)}\right).
\end{aligned}
\label{eq:multiple sum 1rproof}
\end{equation}
The proof of \eqref{e2} is almost identical, except that in the last step the inner sum is performed over $k_1$ rather than $k_r$.
\end{proof}

Following \cite{Ism:zha1} we also consider the following generalization of the ${}_1\psi_1$ function. 
For $\alpha\ge0$, define $B_{q}^{\left(\alpha\right)}$ by 
\begin{equation}
B_{q}^{\left(\alpha\right)}\left(a,b;x\right)=\sum_{n=-\infty}^{\infty}\frac{\left(a;q\right)_{n}}{\left(b;q\right)_{n}}q^{\alpha n^{2}}x^{n}.\label{eq:multiple sum 10}
\end{equation}
Note that $B_q^{(\alpha)}$ is a bilateral series analogue of $A_q^{(\alpha)}$ (with one additional parameter in fact as a result of the generality afforded by the $_1\psi_1$ formula). Again using Lemma \eqref{new1} we obtain the following result.

\begin{thm} We have 

\begin{equation}
\begin{aligned}&B_{q^{r}}^{\left(r\alpha\right)}\left(a^{r},b^{r};x^{r}\right)  =\frac{\left(b,q/a;q\right)_{\infty}^{r}}{\left(q,b/a;q\right)_{\infty}^{r}}\frac{\left(q^{r},b^{r}a^{-r};q^{r}\right)_\infty}{\left(b^{r},q^{r}a^{-r};q^{r}\right)_{\infty}}\\
 & \times\sum_{k_1,\dots,k_{r-1}=-\infty}^{\infty}\frac{\left(a;q\right)_{k_1}\cdots\left(a;q\right)_{k_{r-1}}\zeta_r^{\sum_{i=1}^{r-1}ik_i}q^{\alpha\left(\sum_{i=1}^{r-1} k_i\right)^{2}}x^{\sum_{k=1}^{r-1}k_i}}{\left(b;q\right)_{k_1}\cdots\left(b;q\right)_{k_{r-1}}}B_{q}^{\left(\alpha\right)}\left(a,b;xq^{2\alpha\left(\sum_{i=1}^{r-1}k_i\right)}\right).
\end{aligned}
\label{eq:multiple sum 12r}
\end{equation}
\end{thm}

\begin{proof}
The proof follows from the following series identities 
\begin{equation}
\begin{aligned}&\frac{\left(q,b/a;q\right)_{\infty}^{r}}{\left(b,q/a;q\right)_{\infty}^{r}}\frac{\left(b^{r},q^{r}a^{-r};q^{r}\right)_{\infty}}{\left(q^{r},b^{r}a^{-r};q^{r}\right)_\infty}B_{q^{r}}^{\left(r\alpha\right)}\left(a^{r},b^{r};x^{r}\right)\\
&=\frac{\left(q,b/a;q\right)_{\infty}^{r}}{\left(b,q/a;q\right)_{\infty}^{r}}\frac{\left(b^{r},q^{r}a^{-r};q^{r}\right)_{\infty}}{\left(q^{r},b^{r}a^{-r};q^{r}\right)_\infty}\sum_{m=-\infty}^\infty \frac{(a^r;q^r)_m}{(b^r;q^r)_m}
 q^{\alpha r^2 m^2}x^{rm}\\
&=\sum_{n=-\infty}^\infty \sum_{(k_1,\dots,k_r)\in C_r(n)} \frac{(a;q)_{k_1}\cdots (a;q)_{k_r}}{(b;q)_{k_1}\cdots(b;q)_{k_r}} \zeta_r^{\sum_{i=1}^r ik_i}x^n q^{\alpha n^2}\\
&=\sum_{k_1,\dots,k_{r-1}=-\infty}^{\infty}\frac{\left(a;q\right)_{k_1}\cdots\left(a;q\right)_{k_{r-1}}\zeta_r^{\sum_{i=1}^{r-1}ik_i}q^{\alpha\left(\sum_{i=1}^{r-1} k_i\right)^{2}}x^{\sum_{k=1}^{r-1}k_i}}{\left(b;q\right)_{k_1}\cdots\left(b;q\right)_{k_{r-1}}}B_{q}^{\left(\alpha\right)}\left(a,b;xq^{2\alpha\left(\sum_{i=1}^{r-1}k_i\right)}\right).
\end{aligned}
\label{eq:multiple sum 12rproof}
\end{equation}
\end{proof}

\begin{cor}
The following family of bilateral Rogers-Ramanujan type identities hold

\begin{equation}
\begin{aligned}\sum_{n=-\infty}^{\infty}\frac{q^{r^2n^{2}}x^{rn}}{1-a^{r}q^{rn}} & =\frac{\left(q^{r};q^{r}\right)_{\infty}}{\left(q;q\right)_{\infty}^{r}}\frac{\left(a,q/a;q\right)_{\infty}^{r}}{\left(a^{r},q^{r}a^{-r};q^{r}\right)_{\infty}}\\
 & \times\sum_{k_1,\dots,k_r=-\infty}^{\infty}\frac{\zeta_r^{\sum ik_i}q^{\left(\sum k_i\right)^{2}}x^{\sum k_i}}{\left(1-aq^{k_1}\right)\cdots\left(1-aq^{k_r}\right)}.
\end{aligned}
\label{eq:multiple sum 15r}
\end{equation}

\end{cor}

\begin{proof}
Setting $\alpha=1$ and $b=aq$ and using \eqref{eq:multiple sum 12r} we see that
\begin{equation}
\begin{aligned}
&B_{q^r}^{(r)}(a^r,a^rq^r;x^r)=\frac{\left(q,q;q\right)_{\infty}^{r}}{\left(qa,q/a;q\right)_{\infty}^{r}}\frac{\left(q^{r}a^r,q^{r}a^{-r};q^{r}\right)_{\infty}}{\left(q^{r},q^r;q^{r}\right)_\infty} (1-a^r)\sum_{m=0}^\infty \frac{ q^{\alpha r^2m^2}x^{rm}}{(1-a^rq^m)}\\
&=(1-a)^r \sum_{k_1,\dots,k_r=-\infty}^\infty \frac{ \zeta_r^{\sum_{i=1}^r ik_i}x^{\sum k_i} q^{\alpha (\sum k_i)^2}}{(1-aq^{k_1})\cdots (1-aq^{k_r})},\\
\end{aligned}
\label{eq:multiple sum 13r}
\end{equation}
and the result follows on simplification.
 \end{proof}

\section{Heine Transforms and related identities}
We consider the basic hypergeometric function defined by
\bea
{}_2\phi_1  \left(\left. \begin{matrix} 
 a, b \\
c
\end{matrix}\, \right|q, z \right):=\sum_{k=0}^\infty \frac{(a;q)_k(b;q)_k}{(c;q)_k (q;q)_k} z^k.
\eea
We shall study the function

\bea\label{F1}
F(a,c;z):=\sum_{k=0}^\infty \frac{(a;q)_k (-1)^k q^{k(k-1)/2}z^k}{(q,c;q)_k}= \lim_{b \to \infty}
 {}_2\phi_1  \left(\left. \begin{matrix} 
 a, b \\
c
\end{matrix}\, \right|q,\frac{z}{b} \right). 
\eea
where we have utilized the limit 
\[
\lim_{b\to \infty} \frac{(b;q)_k}{b^k}=(-1)^k q^{k(k-1)/2}.
\]
We have the following alternative representation of $F(a;c;z)$
\begin{thm}
For $c, z\neq q^{-m}$ ($m$ nonnegative integer) we have
\bea \label{F2}
F(a,c,z)=\frac{(z;q)_\infty}{(c;q)_\infty}\sum_{k=0}^\infty\frac{(az/c;q)_k(-1)^k q^{k(k-1)/2}c^k}{(q;q)_k(z;q)_k}
\eea
\end{thm}

\begin{proof}
Recall  the Heine transformation \cite[(III.2)]{Gas:Rah}   
\bea
{}_2\phi_1  \left(\left. \begin{matrix} 
 A , B \\
C
\end{matrix}\, \right|q,z\right) 
&=& \frac{(C/B, Bz;q)_\infty}{(C, z;q)_\infty} 
\; {}_2\phi_1  \left(\left. \begin{matrix} 
 ABz/C , B \\
Bz
\end{matrix}\, \right|q,\frac{C}{B}\right).  
\label{eqHeine2}
\eea

\bea
\notag
\sum_{k=0}^\infty \frac{(a;q)_k (-1)^k q^{k(k-1)/2}z^k}{(q;c;q)_k}&=& \lim_{b \to \infty}
 {}_2\phi_1  \left(\left. \begin{matrix} 
 a, b \\
c
\end{matrix}\, \right|q,\frac{z}{b} \right) \\
\notag
&=&\lim_{ b \to \infty}\frac{(c/b;q)_\infty (z;q)_\infty}{(c;q)_\infty(z/b;q)_\infty}
 {}_2\phi_1  \left(\left. \begin{matrix} 
 az/c, b \\
z
\end{matrix}\, \right|q,\frac{c}{b} \right) \\ &=& \frac{(z;q)_\infty}{(c;q)_\infty}\sum_{k=0}^\infty\frac{(az/c;q)_k(-1)^k q^{k(k-1)/2}c^k}{(q;q)_k(z;q)_k}
\notag
\eea
For the transformation formula to be valid we need neither $c$ nor $z$ to be of the form $q^{-m}$. This proves \eqref{F2}.
\end{proof}
\noindent {\bf Remark.} Note that for $c=z$, both sides of \eqref{F2} are identical. Otherwise the result above provides a way of interchanging the roles of $c$ and $z$ in which both sums are very similar except that the right hand side acquires the factor $\frac{(z;q)_\infty}{(c;q)_\infty}$ and the expression $(a;q)_k$ is replaced by  $(az/c;q)_k$. 

In the following corollary, we examine a case in which the above Theorem leads to transforming infinite sums into finite ones.

\begin{cor}
Let $\alpha, \gamma$ be real numbers such that none of $\alpha+\gamma$ and $\gamma-n$ are negative integers for any nonnegative integer $n$ (in particular, $\gamma$ itself can not be a negative integer). Then the following identity holds

\bea\label{F3}
\sum_{k=0}^\infty \frac{(q^\alpha;q)_k (-1)^k q^{k(k-1)/2}q^{k(\gamma-n)}}{(q,q^{\alpha+\gamma};q)_k}&\\
=  \frac{(q^{\gamma-n};q)_\infty}{(q^{\alpha+\gamma};q)_\infty}\sum_{k=0}^n\frac{(q^{-n};q)_k(-1)^k q^{k(k-1)/2}q^{k(\gamma+\alpha)}}{(q;q)_k(q^{\gamma-n};q)_k},
&
\notag
\eea
\end{cor}
\begin{proof}
This follows directly from \eqref{F2} by setting $a=q^\alpha$ and $c=q^\gamma$ to obtain
\bea
F(q^\alpha, q^{\alpha+\gamma};q^{\gamma-n})= \frac{(q^{\gamma-n};q)_\infty}{(q^{\alpha+\gamma};q)_\infty}\sum_{k=0}^n\frac{(q^{-n};q)_k(-1)^k q^{k(k-1)/2}q^{k(\gamma+\alpha)}}{(q;q)_k(q^{\gamma-n};q)_k}.
\eea

\end{proof}
A well-known consequence of the $q$-binomial theorem  \eqref{qbinom} for $|x|<1$ (see p. 490 of \cite{And:Ask:Roy} for instance)
\bea\label{qbinom1}
\frac{1}{(x;q)_n}=\sum_{k=0}^\infty \frac{(q;q)_{n+k-1}}{(q;q)_k(q;q)_{n-1}}x^k.
\eea
Combining \eqref{qbinom1} with our results above yields the following evaluation.

\begin{cor}
For $|x|<1$ and integer $n\geq 0$ we have
\bea\label{F4}
\sum_{k=0}^\infty \frac{(q;q)_{n+k-1}}{(q;q)_k(q;q)_{n-1}}x^k=\sum_{k=0}^n\frac{(q^{-n};q)_k(-1)^k q^{k(k-1)/2}}{(q;q)_k(xq^{-n};q)_k}x^k.
\eea
\end{cor}

\begin{proof}
By setting $\alpha=0$ in \eqref{F3} the left hand side simplifies to $1$ since $(1;q)_0=1$ and $(1;q)_k=0$ for $k\geq 1$. We thus obtain
\bea\label{F5}
1&=&  \frac{(q^{\gamma-n};q)_\infty}{(q^{\gamma};q)_\infty}\sum_{k=0}^n\frac{(q^{-n};q)_k(-1)^k q^{k(k-1)/2}q^{k\gamma}}{(q;q)_k(q^{\gamma-n};q)_k}.
\eea
Note that 
\bea \label{F6}
\frac{(x;q)_\infty}{(xq^{-n};q)_\infty}=\frac{1}{(xq^{-n};q)_n}.
\eea
Thus, setting $x=q^\gamma$ in \eqref{F5} and applying \eqref{qbinom1} we obtain \eqref{F4}.
\end{proof}

In fact, we can generalize the argument above to obtain the following family of formulas which exhibits a remarkable symmetry between two finite sums of different lengths

\begin{thm}
Let $m,n$ be nonnegative integers, and assume $x\neq q^l$ for any integer $l$, then the following holds 
\bea\label{F7}
\sum_{k=0}^m \frac{(q^{-m};q)_k (-1)^kq^{k(k-2n-1)/2} x^k}{(q,x/q^m;q)_k}=\frac{(xq^{-n};q)_\infty}{(xq^{-m};q)_\infty}\sum_{k=0}^n  \frac{(q^{-n};q)_k (-1)^k q^{k(k-2m-1)/2} x^k}{(q,x/q^n;q)_k}.
\notag
\eea
\end{thm}

\begin{proof}
This follows by setting $\alpha=-m$ in \eqref{F3} and noting that $(q^{-m};q)_k=0$ for $k\geq m+1$. It is well known that the expression $(y;q)_\infty$ with $|q|<1$ converges, and since both sums are now finite there are no convergence concerns, except that we need to guarantee that $(x/q^m;q)_k$ (as well as $(x/q^n;q)_k$ and $(x/q^m;q)_\infty$ are nonzero), which leads to the restrictions placed on $x$.  
\end{proof}

\noindent{Ahmad El-Guindy}\\
{Current address: Science Program, Texas A\&M University in Qatar, Doha, Qatar}\\
{Permanent address: Department of Mathematics, Faculty of Science, Cairo University, Giza, Egypt 12613}\\
email: {a.elguindy@gmail.com}
  \bigskip
  
\noindent M. E. H. I,
  University of Central Florida, 
Orlando, Florida 32828, \\
and King Saud University,  Riyadh, Saudi Arabia\\
  email: ismail@math.ucf.edu


\begin{thebibliography}{99}
 %
 \bibitem{And2}
 G. E. Andrews, {\it The Theory of Partitions}, Addison-Wesley,  Reading, MA, 1976. 
 %
\bibitem {And}  G. E. Andrews, 
  {\it $q$-series: Their development and application in analysis, number theory, combinatorics, physics, and computer algebra}, CBMS Regional Conference Series, 
Number 66,  American Mathematical Society, 
 Providence, RI, 1986. 
 %
 \bibitem{And4}
 G. E. Andrews,
 A polynomial identity which implies the Rogers--Ramanujan identities,  Scripta Math. {\bf  28} (1970),
 297--305. 
%
 \bibitem{And:Ask:Roy}  G. E. Andrews, R. A. Askey, and R. Roy,
{\it Special Functions}, Cambridge University Press, Cambridge, 1999.
%
%

%
 \bibitem{Gar:Ism:Sta}  K. Garrett,  M. E. H. Ismail, and D. Stanton, Variants of the 
 Rogers--Ramanujan identities,
 Advances in Applied Math. {\bf  23} (1999), 274--299. 
%
\bibitem{Gas:Rah} G. Gasper
and M. Rahman, {\it Basic Hypergeometric Series}, second edition,
Encyclopedia of Mathematics and Its Applications, volume 96
Cambridge University Press, Cambridge, 2004.
%
%

%
\bibitem{Ismbook} M. E. H. Ismail, {\it Classical and Quantum Orthogonal
Polynomials in one Variable}, paperback edition, Cambridge University Press,
Cambridge, 2009.
%
%
 
%
%
%
\bibitem{Ism:zha1} M. E. H. Ismail and R. zhang, $q$-Bessel functions and Rogers-Ramanujan type identities, to appear in Proc. Amer. Math. Soc.


%
%
 
%
%

%

%
\bibitem{Sla} L. J. Slater, {\it Generalized Hypergeometric Series}, Cambridge University Press, Cambridge,  1964.
%

%
%

%
%
%
%
%
%


%
%
%
%
\bibitem{Ram}S. Ramanujan, {\it The Lost Notebook and Other Unpublished Papers}
(Introduction by G. E. Andrews), Narosa, New Delhi, 1988.

%
%


\end{thebibliography}
\end{document}